\documentclass[12pt,reqno]{amsart} 

\PassOptionsToPackage{hyphens}{url}\usepackage[hidelinks]{hyperref}

\usepackage[utf8]{inputenc}
\usepackage[british]{babel}
\usepackage{amsmath}
\usepackage{amsthm}
\usepackage{amssymb}
\usepackage{mathrsfs}
\usepackage{indentfirst}
\usepackage{stmaryrd}
\usepackage{bbm}
\usepackage{xspace, xcolor}
\usepackage{thmtools}
\usepackage{thm-restate}
\usepackage{calc}
\usepackage{fullpage}

\usepackage{bookmark}
\usepackage{enumitem}
\usepackage{xcolor}
\usepackage[normalem]{ulem}
\usepackage{comment}
\usepackage[noabbrev, nameinlink]{cleveref}

\usepackage[numbers,longnamesfirst]{natbib}
\usepackage{graphicx}

\usepackage{graphicx}
\graphicspath{{./Pictures/} }
\usepackage{subcaption}

\usepackage{fullpage}
\usepackage{setspace}

\newtheorem{theorem}{Theorem}[section]
\newtheorem{lemma}[theorem]{Lemma}
\newtheorem{corollary}[theorem]{Corollary}
\newtheorem{conjecture}[theorem]{Conjecture}

\newtheorem{proposition}[theorem]{Proposition}

\theoremstyle{definition}

\newtheorem{example}[theorem]{Example}

\newcommand{\arrow}{\rightarrow}
\def\arrows{\rightarrow}
\def\narrows{\nrightarrow}

\newcommand*{\ie}{\text{i.e.}}

\newcommand*{\N}{\mathbb{N}}

\newcommand{\ceil}[1]{\left\lceil #1 \right\rceil}
\newcommand{\floor}[1]{\left\lfloor #1 \right\rfloor}
\newcommand{\ceiln}[1]{\lceil #1 \rceil}

\newcommand{\ceilb}[1]{\big\lceil #1 \big\rceil}

\newcommand{\ceilbg}[1]{\bigg\lceil #1 \bigg\rceil}
\newcommand{\floorbg}[1]{\bigg\lfloor #1 \bigg\rfloor}
\newcommand{\ceilBg}[1]{\Bigg\lceil #1 \Bigg\rceil}

\renewcommand{\leq}{\leqslant}
\renewcommand{\geq}{\geqslant}

\title{Degree conditions for Ramsey goodness of paths}

\author{Lucas Arag\~ao, Jo\~ao Pedro Marciano, Walner Mendon\c{c}a}

\address{IMPA, Estrada Dona Castorina 110, Jardim Bot\^anico, Rio de Janeiro, RJ, Brazil}
\email{\{l.aragao, joao.marciano\}@impa.br}

\address{Instituto de Matem\'atica e Estat\'istica, Universidade de S\~ao Paulo, S\~ao Paulo, Brazil}
\email{walner@ime.usp.br}

\thanks{This study was financed in part by the Coordenação de Aperfeiçoamento de Pessoal de Nível Superior, Brasil (CAPES)}

\begin{document}

\begin{abstract}
    A classical result of Chvátal implies that if $n \geq (r-1)(t-1) +1$, then any  colouring of the edges of $K_n$ in red and blue contains either a monochromatic red $K_r$ or a monochromatic blue $P_t$.
    We study a natural generalization of his result, determining the exact minimum degree condition for a graph $G$ on $n = (r - 1)(t - 1) + 1$ vertices which guarantees that the same Ramsey property holds in $G$. 
    In particular, using a slight generalization of a result of Haxell, we show that $\delta(G) \geq n - \ceil{t/2}$ suffices, and that this bound is best possible.
    We also use a classical result of Bollobás, Erd\H{o}s, and Straus to prove a tight minimum degree condition in the case $r = 3$ for all $n \geq 2t - 1$.
\end{abstract}

\maketitle

\section{Introduction}

We say that a graph $G$ is Ramsey for the pair of graphs $(F, H)$, and write $G \arrows \big(F, H\big)$, if in every red-blue colouring of the edges of $G$, there is either a red monochromatic copy of $F$, or a blue monochromatic copy of $H$. 
When $F=H$, we write simply $G \arrows H$. 
With this notation, the Ramsey number $r(F, H)$ is defined to be the minimum $n$ such that $K_n \arrows \big(F, H\big)$.

For any connected graph $H$ the complete $(r-1)$-partite graph with each part of size $v(H)-1$ implies that
\[r(K_r,H) \geq (v(H)-1)(r-1)+1.\]
A graph $H$ is said to be $r$-good if $r(K_r,H) = (v(H)-1)(r-1)+1$.
Motivated by a result of \citet*{chvatal1977tree}, which implies that trees are $r$-good for any $r \in \N$, the study of Ramsey goodness for general graphs $H$ was initiated by \citet*{burr1983generalizations} in 1983. 
In addition to proving that many sparse graphs are $r$-good, they proposed a number of related problems, several of which remained open for over 25 years.

Since then, Ramsey numbers for dense versus sparse graphs have been extensively studied.
Notably, \citet*{nikiforov2009ramsey} solved all but one of \citeauthor{burr1983generalizations}'s problems.
The main conjecture of \citep{burr1983generalizations}, that all sufficiently large connected graphs with bounded average degree are $r$-good, was disproved by \citet{brandt1996expanding}, but \citet*{allen2013ramsey}, among other very interesting results, proved that, for every $r \in \N$, every connected bounded-degree graph $H$ with bandwidth $o(v(H))$ is $r$-good.
Another well known result is due to \citet{fiz2014ramsey}, who proved that the hypercube is $r$-good for any fixed $r \in \N$ as long as its dimension is sufficiently large.
There are many other recent related results in different directions of this area (see, for example, \citep{pokrovskiy2020ramsey, fox2021ramsey, haslegrave2023ramsey, montgomery2023ramsey}). 
In particular, \citet{pokrovskiy2017ramsey} proved a natural and important generalization of \citeauthor{chvatal1977tree}'s result for paths, determining $r(H,P_n)$ for any $n \geq 4v(H)$, and \citet{balla2018ramsey} extended this generalization to bounded degree trees.
\citet{moreira2019ramsey} studied Ramsey goodness in the random graph $G(n,p)$, proving a sparse random analogue of \citeauthor{chvatal1977tree}’s result for paths, and \citet{araujo2023ramsey} proved a similar theorem for all bounded-degree trees.

Another recent line of work is built on the following question.
For a graph $G$ with at least $r(F, H)$ vertices, what minimum degree condition guarantees that $G \arrows \big(F,H\big)$?
The first person to ask this question in the symmetric case was \citet{Schelp2012}.
More precisely, he asked for which graphs $H$ there exists a constant $0<c<1$ such that, for a graph $G$ with $r(H)$ vertices, if $$\delta(G)\ge c v(G)$$ then $G \arrows H$ (Question 22 in \cite{Schelp2012}). 

This phenomenon of attaining Ramsey properties with only a dense graph, as opposed to a complete graph, has been extensively explored, particularly in the context of paths and cycles.
For instance, \citet{Nikiforov2008}, in one of the pioneering results related to this topic, showed that if $G$ is a graph with $n$ vertices then
\[\delta(G)\geq (1-10^{-7})n \quad \Rightarrow \quad G \arrows C_{\ceil{n/2}}.\]
More generally, they actually proved that, with the above minimum degree condition, there are all cycles up to length $\ceil{n/2}$ in one of the colours\footnote{Note that this result is best possible, even in $K_n$, since one of the colours might be a complete bipartite graph.}.
Observe that the largest monochromatic cycle we can expect to find is of length $\frac{2n}{3}$, since a classical result of \citet{faudree1974all} states that $r(C_{2k}) = 3k-1$ for all $k \in \N$.
The following colouring given by \citet{Schelp2012} shows that there is a graph with $n$ vertices and minimum degree about $\frac{3n}{4}$ which is not Ramsey for $C_{\ceiln{\frac{n}{2}}+1}$.
Let $A_1,A_2,B_1,B_2 \subset [n]$ be a balanced partition of $[n]$, then connect in blue all the pairs in $A_1\cup A_2$ and $B_1\cup B_2$ and connect in red the edges between $A_1$ and $B_1$ and between $A_2$ and $B_2$.
Building on a series of results by various authors \citep{Nikiforov2008, Li2010, GS_2012, benevides2012monochromatic},
\citet*{balogh2022monochromatic} proved that the above construction is tight, \ie, 
$$\delta(G)\geq \frac{3n-1}{4} \quad \Rightarrow \quad G \arrows C_{\frac{2n}{3}}.$$

There are several variations of \citeauthor{Schelp2012}'s question that were introduced very recently.
For instance, \citeauthor{Luczak2019} addressed the three-coloring version for cycles in \cite{Luczak2019} and \cite{Luczak2020} and \citet*{Zhang2023} asymptotically solved the bipartite version for cycles.

We solve an asymmetric version of \citeauthor{Schelp2012}'s question for the pair $(K_r,P_t)$, generalising \citeauthor{chvatal1977tree}'s result for paths and cliques.
More precisely, 
we prove a sharp minimum degree condition that forces a graph $G$ with $r(K_r,P_t)$ vertices to be Ramsey for $(K_r,P_t)$, where $P_t$ is the path with $t$ vertices.
Recall that $r(K_r,P_t) = (r-1)(t-1)+1$.

\begin{theorem}\label{thm:Kr}
    Let $r,t \in \N$, and let $G$ be a graph with $n \geq (r-1)(t-1)+1$ vertices. 
    If 
    \[\delta(G) \geq n - \ceil{\frac{t}{2}}\]
    then $G \arrows \big(K_r,P_t\big)$.
\end{theorem}

Note that the case $r=2$ follows from Dirac's Theorem.
In \Cref{sec:example} we will show that the minimum degree in \Cref{thm:Kr} is tight when $n=r(K_r,P_t)$.
For larger values of $n$, however, we expect that a weaker lower bound on the minimum degree suffices.
In particular, in the case $r=3$ we have the following improved minimum degree condition, which is tight for all $n\geq 2t-1$.

\begin{theorem}\label{thm:k3-more-vertices}
    Let $t \in \N$, let $G$ be a graph with $n$ vertices, and let $k \in \N$ be such that ${2(t-1)k < n \leq 2(t-1)(k+1)}$.
    If
    \[ \delta(G) \ge \floorbg{\frac{n}{2}}+\floor{\frac{1}{k+1} \ceilbg{\frac{n}{2}} }\]
    then $G \arrow \big(K_3,P_t\big)$.
\end{theorem}

In particular, \Cref{thm:k3-more-vertices} gives a second proof of the case $r=3$ of \Cref{thm:Kr}.
In \Cref{sec:example} we will construct, for all $k(r-1)(t-1) < n \leq (k+1)(r-1)(t-1)$, a graph $G$ with $n$ vertices and minimum degree 
\[n-\ceilBg{\frac{k}{k+1}\ceil{\frac{n}{r-1}}}-1\] 
such that ${G\not\arrow (K_r,P_t)}$.
For $r=3$ this construction shows that the lower bound on the minimum degree required in \Cref{thm:k3-more-vertices} is tight.
This construction motivates us to state the following conjecture, which is an extension of \Cref{thm:k3-more-vertices}, replacing the triangle with a arbitrary clique $K_r$. 
The case $r=2$ follows from standard techniques (see \Cref{prop:min-deg-linear-path}) and the case $r=3$ is exactly \Cref{thm:k3-more-vertices}.

\begin{conjecture}\label{conj:main}
    Let $r,t \in \N$ with $r \geq 2$, let $G$ be a graph with $n$ vertices, and let $k \in \N$ be such that ${(r-1)(t-1)k<n\leq (r-1)(t-1)(k+1)}$. 
    If
    \[ \delta(G) \geq n-\ceilBg{\frac{k}{k+1}\ceil{\frac{n}{r-1}}}\]
    then $G \arrow \big(K_r,P_t\big)$.
\end{conjecture}

\subsection{Overview of the proof}

The first step in the proof of both of our main theorems is the following structural result.
In any colouring of the edges of $G$ with no blue $P_t$ we can find a partition $V(G)=A_1\cup \cdots \cup A_m$ such that for any $i \neq j$ there are no blue edges in $G[A_i,A_j]$ and
\[\floor{t/2}+1 \leq |A_i|\leq t-1.\]

Note that to find a red $K_r$ in $G$, it is enough to  find a clique in $G$ with at most one vertex in each part, since there are no blue edges in between parts.
This is equivalent to finding a transversal independent set of the $m$-partite auxiliary graph given by the nonedges of $G$ going in between the parts.
(A transversal independent set of an $m$-partite graph is an independent set with at most one vertex in each part.)
Conditions for finding transversal independent sets have been widely studied, and can be seen as generalizations of Hall's theorem. 
In particular, Haxell (see Theorem 3 in \citep{haxell2001note}) gave a tight sufficient condition for finding a transversal independent set of size $r$ in an $r$-partite graph.
In \Cref{sec:hax}, we use this result to give a similar condition for finding a transversal independent set of size $r$ in an $m$-partite graph, for any $m \geq r$ (see \Cref{thm:gen-haxel}).
This condition, which relates to the non-existence of small dominating sets in the graph, can be verified for our auxiliary $m$-partite graph.

The proof of \Cref{thm:k3-more-vertices} follows a similar strategy, but the argument described above for finding a red triangle in this $m$-partite graph does not work.
That is, when the number of vertices of the graph is greater than $2t-1$, the weaker lower bound on the minimum degree is not sufficient to verify the hypothesis of \Cref{thm:gen-haxel}, as we do in the proof of \Cref{thm:Kr}. 
To amend this, we replace this part of the argument by a classical theorem of \citet*{bollobas1974complete} (see \Cref{thm:r-partite-turan}),
which provides a sharp minimum degree condition for finding triangles in balanced $m$-partite graphs.
The authors of \citep{bollobas1974complete} actually proved an unbalanced version of this result,
but to avoid some technicalities in stating the unbalanced version, we apply the balanced version of this result to an appropriate blow-up of the $m$-partite graph given by the edges of $G$ between the parts.
The blow-up will make the graph balanced, and it is enough to find a triangle in this modified graph to find a red triangle in $G$.

The paper is organized as follows.
In \Cref{sec:example} we give examples showing that the minimum degree conditions in Theorems \ref{thm:Kr} and \ref{thm:k3-more-vertices} and \Cref{conj:main} are tight, in \Cref{sec:hax} we use Haxell's theorem to obtain a sufficient condition for finding a transversal independent set of size $r$ in an $m$-partite graph, and in Sections \ref{sec:pf-of-kr} and \ref{sec:pf-of-thm2} we give the proofs of \Cref{thm:Kr} and \Cref{thm:k3-more-vertices}, respectively.

\section{Constructions}\label{sec:example}
We present two constructions showing that the minimum degree conditions in Theorems \ref{thm:Kr} and \ref{thm:k3-more-vertices} cannot be improved.
First, we describe a graph $G$ with $n=(r-1)(t-1)+1$ vertices and $$\delta(G)=n - \ceil{\frac{t}{2}}-1$$ with $G \narrows \big(K_r,P_t\big)$, showing that the minimum degree condition in \Cref{thm:Kr} is tight.

\begin{example}\label{ex:extr-Kr}
    The graph $G$ is obtained by removing a copy of the complete bipartite graph with part sizes $\floor{t/2}$ and $\ceil{t/2}$ from the complete graph $K_n$.
    The blue edges form $r$ vertex-disjoint cliques: the two parts of the bipartite graph we removed from $G$, and $r -2$ cliques of size $t - 1$. 
    The remaining edges of $G$ are all coloured red.
    
    More formally, let \(V_1,\ldots,V_r\) be disjoint sets with \(|V_i| = t-1\) for every \(i \in [r-2] \), \(|V_{r-1}| = \lceil t/2 \rceil \) and \(|V_{r}| = \lfloor t/2 \rfloor \).
    Let $G$ be the graph obtained by joining every pair of vertices in \( V = V_1 \cup \cdots \cup V_r\), except for those pairs $uv$ with \(u \in V_{r-1}\) and \(v \in V_{r}\).
    Note that \(n = v(G) = (r-1)(t-1) + 1\) and \[\delta(G) = (r-2)(t-1)+\floor{t/2}-1 = n - \ceil{t/2}-1.\]
    Now, colour every edge inside $V_i$ blue, for \(i \in [r]\), and colour every other edge of $G$ red.
    Each part $V_i$ has size at most $t-1$, so there is no blue copy of $P_t$.
    Moreover, since the graph induced by the red edges is $(r-1)$-partite with parts $V_1, \ldots, V_{r-2}, V_{r-1}\cup V_r$, this colouring also contains no red copy of $K_r$.
\end{example}

We will next show that for every $(r - 1)(t - 1)k < n \leq (r - 1)(t - 1)(k + 1)$, where $r \geq 3$ and $k,t \geq 1$, there exists a graph $G$ with $n$ vertices and
\[\delta(G)=n-\ceilBg{\frac{k}{k+1}\ceil{\frac{n}{r-1}}} -1,\]
such that $G \narrows \big(K_r,P_t\big)$.
This construction motivates the bound on $\delta(G)$ in \Cref{conj:main}.
\begin{example}\label{ex:many-vtx}
    The graph $G$ is obtained by placing $k+1$ cliques of size $\approx n / (r-1)(k+1)$ into each part of the Turan graph $T_{r-1}(n)$.
    The edges of these cliques are coloured blue, and the remaining edges of $G$ are coloured red.
    
    More formally, we will define a graph $G$ with $n$ vertices as follows.
    Let $$V(G)=\bigcup_{i =1}^{r-1} A_{i}$$ be a partition of $V(G)$ into $r-1$ parts of size that differ at most by one.
    Now for each $i \in [r-1]$ define $$A_i=\bigcup_{j = 1}^{k+1}A_{i,j}$$ a partition of $A_i$ into $k+1$ parts of size that differ at most by one.
    Two vertices $v\in A_{i_1,j_1}$ and $w\in A_{i_2,j_2}$ are connected in $G$ if $i_1\neq i_2$ and also if $i_1=i_2$ and $j_1=j_2$.
    The edge $vw$ is coloured red if $i\neq j$, otherwise is coloured blue.
    
    It is easy to see that the minimum degree of $G$ is given by the vertices that are in parts $A_{i,j}$ with $|A_i|=\ceil{\frac{n}{r-1}}$ and $|A_{i,j}|=\floor{\frac{|A_i|}{k+1}}$.
    Thus,
    \[\delta(G)=n-\ceil{\frac{n}{r-1}}+\floor{\frac{1}{k+1}\ceil{\frac{n}{r-1}}}-1 = n- \ceilBg{\frac{k}{k+1}\ceil{\frac{n}{r-1}}} -1.\]
    Since $1\leq i \leq r-1$, the largest red clique has size $r-1$, and the largest blue path has size at most $|A_{i,j}|\leq \ceil{ \frac{1}{k+1}\ceil{\frac{n}{r-1}}} \leq t-1$.
\end{example}
Note that in \Cref{ex:many-vtx} the blue cliques can be as small as $t/2$.

\section{Transversal Independent Sets}\label{sec:hax}

In this section we will use a result of Haxell to prove \Cref{thm:gen-haxel}, below. 
Recall that a transversal independent set in an $m$-partite graph is an independent set with at most one vertex in each part. 
The theorem gives a sufficient condition for an $m$-partite graph to contain a transversal independent set of size $r$. 
A set of vertices $A$ in a graph $G$ is \emph{dominated} by another set of vertices $B$ if $A \subseteq N_G(B)$ (that is, every vertex in $A$ has a neighbour in $B$).

\begin{theorem}\label{thm:gen-haxel}
  Let $m \geq r \geq 1$, let $G$ be a graph, and let $V(G) = V_1 \cup \cdots \cup V_m$ be a partition of the vertex set.
  Suppose that for every $S \subseteq [m]$, the set of vertices
  \[
    V_S = \bigcup_{i \in S} V_i
  \]
  is not dominated by any set of \(2\big(|S|-m+r-1\big)\) vertices of $V_S$.
  Then there exists an independent set $I$ with \(|I| \geq r\), such that \(\big|I \cap V_i\big| \leq 1\) for every \(i \in [m]\).
\end{theorem}

A result equivalent to the case $m = r$ of \Cref{thm:gen-haxel} follows from the proof in \citep{haxell1995condition} and is stated explicitly in \citep{haxell2001note}.
For a topological proof of a more general result, see \citep{meshulam2001clique}, and for a simple combinatorial proof, see the survey \citep{haxell2011forming}.
We will deduce the general case from the case $m=r$.
The proof uses a standard technique for deducing such ``defect versions'' of variants of Hall's theorem (see for instance \citep{aharoni2001ryser} for a similar approach), but for the reader's convenience we give the details.

\begin{proof}
    Define $G'$ to be the disjoint union of $G$ with $k$ disjoint copies of $K_m$, which we denote by $T_1, \ldots, T_k$.
    For each $i$, we place exactly one vertex of $T_i$ in each of the parts $V_1, \ldots, V_m$.
    More precisely, let $V(T_i)=\big\{u_1(i), \ldots, u_{m}(i)\big\}$ for each $i \in [k]$, and define, for each $j \in [m]$,
    \[V_j' = V_j\cup\big\{u_j(i) \mid  i \in [k]\big\}.\]
     
    Note that a transversal independent set of size $m$ in $G'$ cannot contain more than one vertex from $V(T_i)$ for each $i \in [k]$, since $T_i$ is clique.
    Therefore, a transversal independent set of size $m$ in $G'$ contains a transversal independent set of size $m-k=r$ in $G$.
    Thus, it is enough to show that the $m$-partite graph $G'$ satisfies the conditions of \Cref{thm:gen-haxel} with $r=m$.
    That is, we need to show, for every $S\subset [m]$, that the set
    \begin{equation}\label{eq:hyp-for-G'}
        V_S':=\bigcup_{i\in S}V_i'
    \end{equation}
    is not dominated by any set of $2(|S|-1)$ vertices of $V_S'$.
    
    Fix $S\subset [m]$, let $X \subset V_S'$ with $|X|=2(|S|-1)$ and assume for a contradiction that $X$ dominates $V_S'$.
    Observe that $X$ must contain at least $2$ vertices from $V(T_i)$ for each $i \in [k]$, since $T_i$ is a connected component of $G'$.
    Therefore,
    \begin{equation}\label{eq:X-size}
        \big|X\cap V_S\big| = |X|- \sum_{i=1}^k\big|X \cap V(T_i)\big| \leq 2(|S|-1) - 2k = 2\big(|S|-m+r-1\big).
    \end{equation}
    Note also that $X \cap V_S$ dominates $V_S$, since a vertex in $X \cap V(T_i)$ has no neighbours in $V_S$.
    Together with \eqref{eq:X-size}, this contradicts our assumption that $V_S$ is not dominated by any set of $2\big(|S| - m + r - 1\big)$ vertices of $V_S$, and this contradiction proves \eqref{eq:hyp-for-G'}. 
    As explained above, we may therefore apply the theorem of Haxell (the case $m = r$ of Theorem 3.1) to find a transversal independent set of size $r$ in $G$, as required.
\end{proof}

\section{Proof of Theorem~\ref{thm:Kr}}\label{sec:pf-of-kr}

The first step in the proof is the following lemma, which gives the decomposition described in the introduction. 

\begin{lemma}\label{lem:decomposition}
	Let $d \in \N$, and let $G$ be a $P_{d}$-free graph with $n$ vertices and $\delta(G) \geq \floor{d/2}$.
	There exists a partition $V(G) = A_1 \cup \cdots \cup A_m$ of the vertex set such that the following hold for every $i \in [m]$:
	\begin{equation}\label{eq:sizeofA}
		\floor{d/2}+1 \leq |A_i| \leq d-1,
	\end{equation}
    $A_i$ is a connected component of $G$, and $G[A_i]$ has a hamiltonian cycle.
\end{lemma}

We note that somewhat similar results were proved by \citeauthor*{allen2013ramsey} \citep[Lemma 19]{allen2013ramsey} and \citeauthor{moreira2019ramsey} \citep[Proposition 3.2]{moreira2019ramsey}, but neither of these results is suitable for our purposes; the former since it requires a stronger minimum degree condition, and the latter since it does not give a partition of all the vertices.
We will deduce \Cref{lem:decomposition} from the following fundamental lemma of Pósa (see \citep{bollobas1998modern}, Chapter IV, Theorem 2). 
For the reader's convenience, we give a short proof of this result.

\begin{lemma}\label{prop:ham-cycle}
    Let $P$ be a path of maximum length in a graph $G$, let $U = V(P)$, and let $u$ and $v$ be the end vertices of $P$.
	If 
    \begin{equation}\label{eq:lb-endvtx}
         d(u)+d(v) \geq |U|
    \end{equation}
	then $G[U]$ has a hamiltonian cycle. 
    Moreover, there are no edges between $U$ and $V(G) \setminus U$.
\end{lemma}

\begin{proof}
    Let $U=\{u_1, \ldots, u_\ell\}$ be the vertex set of the path $P$ with edge set $\{ u_1u_2,\ldots,u_{\ell-1}u_\ell\}$.
	For $i \in \{1,\ell\}$ we have $N_G(u_i) \subset U$, otherwise we could extend our path, contradicting its maximality.
    If $u_1u_\ell \in E(G)$, then this edge together with $P$ forms a hamiltonian cycle in $G[U]$.
    If not, then by \eqref{eq:lb-endvtx} and the pigeonhole principle, there is some $i\in \{2, \ldots,\ell-2\}$ such that $u_{i+1} \in N_G(u_1)$ and $u_i \in N_G(u_{\ell})$.
	Therefore, the sequence of vertices $$u_1u_{i+1}u_{i+2} \ldots u_{\ell}u_iu_{i-1} \ldots u_1$$ induces a hamiltonian cycle in $G[U]$.
    
    Finally, note that if there were an edge between $U$ and $V(G)\setminus U$ then we could find a path longer than $P$ in $G$, contradicting our choice of $P$.
\end{proof}

Now we proceed to the proof of \Cref{lem:decomposition}.
It is proved by greedily removing maximal paths from the graph and showing that we can close each to a cycle of the same length.

\begin{proof}[Proof of \Cref{lem:decomposition}]
	We will describe directly the $i$-th step of the algorithm.
    Let $i \geq 1$ and assume we have $A_1, \ldots, A_{i-1}$ being connected components of $G$, satisfying \eqref{eq:sizeofA} and that $G[A_j]$ has a hamiltonian cycle for each $j \in [i-1]$.
    Moreover, assume that the set $V_i = V(G) \setminus \big( A_1 \cup \cdots \cup A_{i-1} \big)$ of remaining vertices is non-empty.
    
	Let $A_i$ be the vertex set of a path of maximum length $P_i$ in $G_i:=G[V_{i}]$.
    It is enough to verify that $A_i$ is a connected component of $G$ satisfying \eqref{eq:sizeofA} and that $G[A_i]$ has a hamiltonian cycle.
	Since $\delta(G_i) \geq \delta(G) \geq \floor{d/2}$ and by the assumption that $G_i$ is $P_{d}$-free, we have 
	\begin{equation*}
	\floor{d/2}+1 \leq |A_i| \leq d-1.
	\end{equation*}
    Since $\delta(G_i) \geq \floor{d/2} \geq |A_i|/2$, we can apply \Cref{prop:ham-cycle} to the path $P_i$ in the graph $G_i$ to show that there is a hamiltonian cycle in $G_i[A_i]$ and moreover that there are no edges between $A_i$ and $V_i\setminus A_i$, so $A_i$ is a connected component of $G_i$ (and hence of $G$), as required.
\end{proof}

Finally, let's prove \Cref{thm:Kr}.

\begin{proof}[Proof of \Cref{thm:Kr}]
The proof will proceed by induction on $r\geq 2$.
If $r=2$, we have 
\[\delta(G) \geq n-\ceil{t/2} \geq \floor{n/2} \hspace{-2pt},\]
where the second inequality hold since $n \geq t$,
and every edge of $G$ is blue. 
Thus by Dirac's theorem we can find a hamiltonian path in $G$ and, in particular, a blue $P_t$.

Let $r \geq 3$, and assume the theorem holds for $r - 1$. 
Let $n \geq (r - 1)(t -1) + 1$, let $G$ be a graph with $n$ vertices and $\delta(G) \geq n - \ceil{t/2}$, and let $R$ and $B$ be the graphs of red edges and blue edges, respectively, in a $2$-edge-colouring of $E(G)$.
Assume that this colouring contains no blue copy of $P_t$; we will show it contains a red copy of $K_r$.

Suppose first that there exists $u \in V(G)$ with $d_R(u) \geq n-t+1 \geq (r-2)(t-1)+1$.
We claim that, in this case, 
\[\delta\big(G[N_R(u)]\big) \geq d_R(u) - \ceil{t/2}\hspace{-2pt}.\]
Indeed, for every $v \in N_R(u)$ we have that
\begin{align*}
    \big|N_G(v)\cap N_R(u)\big| &\geq \big|N_G(v)\big| - \big|N_G(v) \setminus N_R(u)\big|\\ 
    &\geq (n-\ceil{t/2}) - (n-d_R(u)) \\
    &= d_R(u)-\ceil{t/2}\hspace{-2pt}.
\end{align*}
Therefore, $G[N_R(u)]$ satisfies the induction hypothesis and it follows that 
$$G[N_R(u)] \rightarrow \big(K_{r-1},P_t\big).$$
But $G$ has no blue $P_t$, so there is a red $K_{r-1}$ in $N_R(u)$.
Together with $u$ this forms a red copy of $K_r$ in $G$, as we wanted to find.

We can therefore assume that $d_R(u) \leq n-t$, and hence,
$$d_B(u) \geq  \delta(G) -d_R(u) \geq n- \ceil{t/2}-(n-t) = t- \ceil{t/2} = \floor{t/2}$$ 
for every $u \in V(G).$
Applying \Cref{lem:decomposition} to the graph $B$ with $d=t$ we obtain a partition $$V(G)=V(B) = A_1 \cup \cdots \cup A_m$$  with 
\begin{equation}
	 \floor{t/2}+1 \leq |A_i| \leq t-1
\end{equation}
for each $i \in [m]$, and hence
\begin{equation}
	r \leq \ceil{\frac{n}{t-1}} \leq m,
\end{equation}
since $n \geq (r-1)(t-1) + 1$.
Moreover, all the edges of $G$ between different parts are red.

Observe that any $K_r$ in $G$ with at most one vertex in each part must be a red $K_r$, which motivates the following definition.
Define $H$ to be the $m$-partite graph with parts $A_1, \ldots,A_m$, in which two vertices are connected in $H$ if they are not an edge of $G$.
In this way, it is enough for us to find an independent transversal set of size $r$ in $H$ to yield a contradiction.

In order to do so, we will show that $H$ with the partition $V(H)=A_1\cup \cdots \cup A_m$ satisfies the conditions of \Cref{thm:gen-haxel}.
Let $S \subset [m]$ and note that
\[\Delta(H) \leq n-1-\delta(G) \leq \ceil{t/2}-1.\]
Let $X \subset A_S$ be any set with $|X| \leq 2(|S|-m+r-1)$, and observe that $X$ dominates at most
\[\Delta(H)\cdot |X| \leq 2\big(|S|-m+r-1\big)\big(\ceil{t/2}-1\big)\]
vertices.
Now, observe that 
\begin{align}
    2\big(|S|-m+r-1\big)\big(\ceil{t/2}-1\big) &\leq \big(r-|S^c|-1\big)(t-1) \label{align:1} \\  
    &\leq n - |S^c|(t-1)-1 \label{align:2}\\
    &< |A_S| \label{align:3},
\end{align}
where \eqref{align:1} holds because $|S|+|S^c|=m$ and $2\ceil{t/2} \leq t+1$, \eqref{align:2} is just our assumption that $n \geq (r-1)(t-1)+1$, and \eqref{align:3} follows because $|A_i| \leq t-1$ for every $i \in S^c$.
It follows that $X$ cannot dominate $A_S$, and therefore, by \Cref{thm:gen-haxel}, there exists an independent transversal set of size $r$ in $H$. 
By the observations above, this is exactly a red copy of $K_r$ in $G$, as required.
\end{proof}

\section{Proof of Theorem~\ref{thm:k3-more-vertices}}\label{sec:pf-of-thm2}

In this section we will prove \Cref{thm:k3-more-vertices}; the proof uses some of the same ideas as that of \Cref{thm:Kr}, but instead of using \Cref{thm:gen-haxel}, we will use a different idea to find a triangle in some blow-up of the red graph.

We will need the following straightforward (and presumably well-known) proposition, which gives a tight lower bound condition on the minimum degree of a graph with $n$ vertices to contain a path with at least $\ceil{n/k}$ vertices. 
It is a simpler variant of an old result of \citet{alon1986longest} (see also \citep{egawa1989longest}), which says that every graph with minimum degree $\ceil{n/(k+1)}$ contains a cycle of length $\floor{n/k}$, but since it follows easily from \Cref{lem:decomposition}, we give the proof for completeness.

\begin{proposition}\label{prop:min-deg-linear-path}
    Let $n,k \in \N$ and let $G$ be a graph with $n$ vertices. If $G$ has minimum degree at least $\floor{\frac{n}{k+1}}$ then it contains a path with $\ceil{\frac{n}{k}}$ vertices.
\end{proposition}

\begin{proof}
    Assume by contradiction that $G$ is $P_{\ceil{n/k}}$-free. 
    We will apply \Cref{lem:decomposition} to the graph $G$ with $d=2\floor{n/(k+1)}+1$.
    The reader can easily verify that 
    \begin{equation}\label{eq:d-ineq}
        \delta(G) \geq \floor{\frac{d}{2}} = \floor{\frac{n}{k+1}} \qquad \text{ and } \qquad  d \geq \ceilbg{\frac{n}{k}} \hspace{-2pt}.
    \end{equation}
    Since $G$ is $P_{\ceil{n/k}}$-free, it follows that $G$ is also $P_d$-free.
    Thus, by \Cref{lem:decomposition}, we obtain a partition $A_1, \ldots, A_m$ of the vertices of $G$ with 
    \[\floor{\frac{n}{k+1}} +1 \leq |A_i| \leq \ceilbg{\frac{n}{k}} -1\]
    for each $i \in [m]$, where the upper bound holds since $G[A_i]$ is hamiltonian and $P_{\ceil{n/k}}$-free.
    Then, using $\ceil{x}-1<x<\floor{x}+1$,
    \[k = \frac{n}{\frac{n}{k}}<\frac{n}{\ceil{\frac{n}{k}} -1} \leq m \leq \frac{n}{\floor{\frac{n}{k+1}} +1} < \frac{n}{\frac{n}{k+1}} = k+1,\]
    yielding the desired contradiction since $m$ must be an integer.
\end{proof}

For the proof of \Cref{thm:k3-more-vertices}, applying \Cref{thm:gen-haxel} does not work.
The reason for that is because, in this case, the condition about the non-existence of small dominating sets is too strong.
Instead, we use the following theorem of \citet*{bollobas1974complete}.
\begin{theorem}\label{thm:r-partite-turan}
    Let $H$ be a balanced $m$-partite graph where each part has size $N$.
    If 
    \[\delta(H) > \floorbg{\frac{m}{2}}N\]
    then $H$ contains a triangle.
\end{theorem}
The authors of \citep{bollobas1974complete} actually prove an unbalanced version of this theorem.
To avoid using this more technical statement, we will use a simple blow-up argument in order to apply the balanced case directly.
We do this in the following simple corollary, which is tailored for our application.

\begin{corollary}\label{cor:r-partite-turan}
    Let $k \in \N$, and let $H$ be an $m$-partite graph with parts $A_1, \ldots, A_m$ where $m=2k+1$.
    If for each $i,j \in [m]$ and $u \in A_i$ we have
    \begin{equation}\label{eq:cond-of-cor}
        \big|N_H(u)^c \setminus A_i\big| < k|A_j|,
    \end{equation}
    then $H$ contains a triangle.
\end{corollary}
\begin{proof}
    We would like to apply \Cref{thm:r-partite-turan} to $H$, but since $H$ is not necessarily balanced, we first need to construct from $H$ an auxiliary graph $H'$ with balanced parts.
    Consider a blow-up of $H$ into a balanced $m$-partite graph $H'$ with parts $A_1',A_2', \ldots,A_m'$ obtained in the following way: for each $i \in [m]$, we replace each vertex $u \in A_i$ by a set $F(u) \subset A_i'$ of $\frac{N}{|A_i|}$ vertices, where $N:=\prod_{i=1}^{m}|A_i|$.
    More precisely, for each $i \in [m]$ \[A_i'= \bigcup_{u \in A_i}F(u)\]
    and 
    \[E(H') = \big\{ xy : x \in F(u) \text{ and } y \in F(v) \text{ for some } uv \in E(H) \big\}.\]
    
    Observe that every triangle in $H'$ corresponds to a triangle in $H$ and therefore, by \Cref{thm:r-partite-turan}, to finish the proof it will suffice to show that
    \[\delta(H') > \floorbg{\frac{m}{2}}N = kN.\]
    Let $M= \min_{j \in [n]} |A_j|$ and observe that $|F(u)| \leq N/M$ for every $u \in V(H)$, and therefore for every $i \in [m]$, $u \in A_i$ and $u_0 \in F(u)$, we have
    \[\big|N_{H'}(u_0)^c\setminus A_i'\big| \leq \frac{N}{M} \big|N_H(u)^c\setminus A_i\big| < kN,\]
    where the second inequality holds from the hypothesis \eqref{eq:cond-of-cor}.
    Therefore, since $|A_i'|=N$,
    \[\delta(H') > mN - kN - |A_i'| = kN,\]
    as we wanted to prove.
\end{proof}

Finally, let's proceed to the proof of \Cref{thm:k3-more-vertices}.

\begin{proof}[Proof of \Cref{thm:k3-more-vertices}]
    Let $n = 2(t -1)k + s$ vertices, where $1\leq s \leq 2(t-1)$, let $G$ be a graph with $n$ vertices and 
    \[\delta(G) \geq \floorbg{\frac{n}{2}}+\floor{\frac{1}{k+1} \ceilbg{\frac{n}{2}}}\hspace{-2pt},\] and let $R$ and $B$ be the graphs of red edges and blue edges, respectively, in a $2$-edge-colouring of $E(G)$.
    Assume that this colouring contains no blue copy of $P_t$; we will show that it contains a red copy of $K_3$.
      
    Suppose first that there exists $u \in V(G)$  with $d_R(u) \geq \ceil{\frac{n}{2}}$, we claim that 
    \begin{equation}\label{eq:min-deg-neighbourhood}
        \delta\big(G[N_R(u)]\big) \geq \floor{\frac{d_R(u)}{k+1}}\hspace{-2pt}.
    \end{equation}
    Indeed, for every $v \in N_R(u)$ we have
    \begin{align*}
        \big|N_G(v) \cap N_R(u)\big| &\geq \big|N_G(v)\big| - \big|N_G(v) \setminus N_R(u)\big| \\
        &\geq  \floorbg{\frac{n}{2}} + \floor{\frac{1}{k+1} \ceilbg{\frac{n}{2}}} - (n-d_R(u))\\
        &= d_R(u) - \ceilbg{\frac{n}{2}} + \floor{\frac{1}{k+1} \ceilbg{\frac{n}{2}}}  \\
        & \geq \floor{\frac{d_R(u)}{k+1}}\hspace{-2pt},
    \end{align*}
    as claimed, where the last inequality uses that the function $\phi(m) = m-\floor{\frac{m}{k+1}}$ is increasing on the positive integers, so that $\phi(d_R(u)) \geq \phi(\ceil{n/2})$, since $d_R(u) \geq \ceil{n/2}$.
    
    Now, by equation \eqref{eq:min-deg-neighbourhood}, we can apply \Cref{prop:min-deg-linear-path} in $G[N_R(u)]$ to obtain a path with at least $\ceilb{\frac{d_R(u)}{k}}$ vertices in $G[N_R(u)]$.
    Using again that $d_R(u) \geq \ceil{n/2}$, $n=2k(t-1)+s$ and $s \geq 1$, we have that the number of vertices of this path is at least
    \[\ceil{\frac{d_R(u)}{k}} \geq \ceilbg{\frac{n}{2k}} = \ceilbg{(t-1)+\frac{s}{2k}} \geq t,\]
    which implies there is a red edge in $G[N_R(u)]$, since we have no blue $P_t$.
    This red edge in $G[N_R(u)]$ together with $u$ forms the desired red copy of $K_3$.
    
    Therefore, we can assume that $\Delta(R) \leq \ceil{\frac{n}{2}}-1$, and hence, 
    \begin{align*}
        \delta(B) \geq  \delta(G)-\Delta(R) \geq & \floorbg{\frac{n}{2}}+\floor{\frac{1}{k+1} \ceilbg{\frac{n}{2}}}-\left(\ceilbg{\frac{n}{2}} - 1\right) \\
        \geq &\floor{\frac{1}{k+1} \ceilbg{\frac{n}{2}}} \hspace{-2pt}, 
    \end{align*}
    since $\lfloor n/2 \rfloor + 1 \geq \lceil n/2 \rceil$.
    We will apply \Cref{lem:decomposition} to the graph $B$ with
    \[d=2\floor{\frac{1}{k+1}\ceilbg{\frac{n}{2}}}+1.\]
    We may do so since  
    \begin{equation}\label{eq:d-over-2}
          \delta(B)\geq \floor{\frac{1}{k+1}\ceilbg{\frac{n}{2}}} = \floor{\frac{d}{2}},
    \end{equation}
    and since $B$ is $P_d$-free, which follows since $n > 2(t-1)k$, and therefore
    \begin{equation*}\label{eq:claim}
        d \geq 2 \floor{ \frac{(t-1)k + 1}{k+1}} + 1 \geq 2 \floor{ \frac{t}{2}}+ 1 \geq t.
    \end{equation*}
    
    Hence, by \Cref{lem:decomposition}, there exists a partition $V(G) = A_1 \cup \cdots \cup A_m$ such that every edge between different parts is red and
    \begin{equation}\label{eq:size-A}
    \floor{\frac{d}{2}}+1 \leq |A_i| \leq t-1
    \end{equation}
    for each $i \in [m]$, where the upper bound holds since $G[A_i]$ is hamiltonian and $P_t$-free.
    
    Define an auxiliary graph $H=R[A_1, \ldots, A_m]$ to be the $m$-partite graph of the red edges between distinct parts in this partition $V(G)=A_1\cup\cdots\cup A_m$.
    Our goal now is to apply \Cref{cor:r-partite-turan} to $H$, which will give us the desired red copy of a triangle.
    Let's verify first that, in fact, $m=2k+1$.
    To do so, observe that, by \eqref{eq:size-A} and our choice of $d$,
    \begin{equation}\label{eq:ineq-sizeofA}
        \frac{n}{2(k+1)}\leq  \frac{1}{k+1}\ceilbg{\frac{n}{2}} <\floor{\frac{d}{2}}+1 \leq |A_i| \leq  t-1 < \frac{n}{2k}
    \end{equation}
    for each $i \in [m]$, where the final inequality holds since $n>2(t-1)k$.
    We therefore have
    $$2k < m < 2(k+1) \quad \Rightarrow \quad  m=2k+1,$$ since $m$ is an integer.

    Now it is enough to verify the bound in \eqref{eq:cond-of-cor} for $H$.
    Let $M= \min_{j \in [m]} |A_j|$ and observe that, by \eqref{eq:ineq-sizeofA}, we have
    \begin{equation}\label{eq:size-of-M}
        M \geq \floor{\frac{d}{2}}+1>\frac{1}{k+1}\ceilbg{\frac{n}{2}}.
    \end{equation}
    Hence, for every $i \in [m]$ and $u \in A_i$ we have
    \begin{align*}
        \big|N_H(u)^c\setminus A_i\big| &\leq (n-|A_i|)-(\delta(G)-|A_i|+1) \\
        &\leq n - \left(\floorbg{\frac{n}{2}}+ \floor{\frac{1}{k+1}\ceilbg{\frac{n}{2}}}+1\right) \\ 
        &< \frac{k}{k+1} \ceilbg{\frac{n}{2}} \leq kM,
    \end{align*}
    as we wanted to prove, where the last inequality follows from \eqref{eq:size-of-M}.
\end{proof}

\section*{Acknowledgements}
We would like to thank Rob Morris for many helpful conversations and comments on the writting of this paper.

\bibliographystyle{plainnat}
\bibliography{main}

\end{document}